\documentclass[11pt,reqno]{amsart}

\usepackage[papersize={8.5in,11in},width=6in,height=8in,centering]{geometry}
\usepackage{amsfonts}
\usepackage{amssymb}
\usepackage{amsthm}
\usepackage{amsmath}
\usepackage{graphicx}

\usepackage[all]{xy}
\usepackage{xcolor}

\newcommand{\Z}{\mathbb{Z}}

\newcommand{\R}{\mathbb{R}}

\theoremstyle{definition}
 \newtheorem{theorem}{\bf Theorem}[section]
 
 \newtheorem{lemma}[theorem]{Lemma}
 \newtheorem{corollary}[theorem]{Corollary}
 
\theoremstyle{definition}
 \newtheorem{example}[theorem]{Example}
 \newtheorem{remark}[theorem]{Remark}
 \newtheorem{definition}[theorem]{Definition}
 \newtheorem{proposition}[theorem]{Proposition}
\numberwithin{equation}{section}

\title[Contact toric 3-manifolds]{Contact 3-manifolds that admit a non-free toric action}

\author[A. Marinkovi\'c]{Aleksandra Marinkovi\'c} 

\email{aleksandra.marinkovic@matf.bg.ac.rs}

\author[L. Starkston]{ Laura Starkston}

\email{lstarkston@math.ucdavis.edu}

\begin{document}

\maketitle

\begin{abstract}

 We classify contact toric 3-manifolds up to contactomorphism, through explicit descriptions, building off of work by Lerman  \cite{Lerman2}. As an application, we classify all contact structures on 3-manifolds that can be realised as a concave boundary of linear plumbing over spheres.
 The later result is inspired by the work \cite{MNRSTW}.
\end{abstract}
\tableofcontents

\section{Introduction}
The study of
contact toric manifolds started  in the work of Banyaga and Molino (\cite{BM92, BM96}) who 
were the first to explore completely integrable systems in the contact case. 
They proved
an analog to Delzant's convexity theorem  \cite{Delzant} for particular contact toric manifolds, called of Reeb type. 
These are contact toric manifolds  that admit an invariant contact form whose corresponding Reeb vector field generates a circle subaction of the toric action. 
 Further, Boyer and Galicki in \cite{BG} extended their results by showing that  these contact toric manifolds are  reductions of odd dimensional spheres, analogous to the result that closed connected symplectic toric manifolds are reductions of complex spaces. 
 Finally, Lerman in \cite{Lerman2}, gave a complete list of contact toric 3-manifolds and classified the underlying smooth 3-manifolds. 
 
 When the toric action is free, Lerman proves the contact toric manifold must be $T^3$ with contact structure $\ker(\cos(2\pi n t)d\theta_1+\sin(2\pi nt)d\theta_2)$, $n\geq 1$. When the action is not free, Lerman's results show that such a contact toric manifold must have a certain form determined by two real numbers $t_1,t_2$ corresponding to angles of rational slope (see Definition~\ref{def:Yt1t2}). 
 The underlying 3-manifold is a lens space (possibly $S^3=L(1,0)$ or $S^1\times S^2=L(0,1)$). 
 
In this article, we complete the classification of contact 3-manifolds that admit a non-free toric action up to contactomorphism. We completely determine when two of the contact manifolds of Lerman's form are contactomorphic or not. We also prove which of these contact structures are tight or overtwisted and give explicit descriptions of these contact manifolds through alternate topological constructions. In the end we have a complete list of contactomorphism types which are realized as contact toric 3-manifolds with no repetition.

We divide our classification up to contactomorphism of contact 3-manifolds admitting a non-free toric action into the tight case and the overtwisted case. Here are the main results.


\begin{theorem} \label{thm:tight}
Given any lens space $L(k,l)$, up to contactomorphism, it has a unique tight contact structure that admits a toric action.

For a lens space $L(k,l)$ different from $S^1 \times S^2$ this tight contact structure is induced from the unique tight contact structure on $S^3$ by a $\mathbb Z_k$-quotient. Thus, this is the universally tight contact structure on $L(k,l)$. 
\end{theorem}

Note that in general, most lens spaces admit non-contactomorphic tight contact structures. Our result shows that these additional (virtually overtwisted) contact structures are not realized as contact toric manifolds.

\begin{theorem} \label{thm:overtwisted}
Up to contactomorphism, there are exactly two overtwisted contact structures on any lens space $L(k,l)$ that admit a toric action. These are obtained by performing a half- and a full-Lutz twist to the unique tight contact structure on $L(k,l)$ that admits a toric action.
\end{theorem}

We remark that there are no analogues to these results in higher dimensions. Namely, in \cite{AM} Abreu and Macarini constructed infinitely many tight contact structures on $S^3\times S^2$ that admit a non-free toric action, while in \cite{M15} it is shown that all contact toric manifolds in higher dimensions are fillable (thus tight).

Our focus on contact toric 3-manifolds was motivated by the work
 \cite{MNRSTW}, where the authors, together with J. Nelson, A. Rechtman, S. Tanny and L. Wang, showed the existence  of such contact toric structures on the concave boundary of certain linear plumbings. 
 
  \begin{theorem} \cite[Theorem 4.1]{MNRSTW}   \label{thm:plumbing_is_toric} The boundary of any linear plumbing over spheres with self-intersection numbers $(s_1, \ldots, s_n)$, where $s_i \geq0,$ for at least one index $i \in \{1, \ldots, n \}$, admits a concave Liouville structure inducing a contact structure admitting a non-free contact toric action.
   \end{theorem}
    
 As an application to this result we were able to detect when the contact structure on the boundary of the plumbing  $(s_1,  \ldots, s_n)$ is tight or overtwisted, simply by exploring the corresponding moment cones that are specified by self-intersection numbers 
 $s_1,  \ldots, s_n$. 
 
   \vskip2mm
    
In this article we prove the converse of Theorem   \ref{thm:plumbing_is_toric}. 
 \begin{theorem} \label{thm:toric_is_plumbing}
Any contact 3-manifold that admits a non-free toric action can be realised as a concave contact  boundary of a linear plumbing of spheres with self-intersection numbers $s_1, \ldots, s_n$ where $s_i \geq0,$ for at least one index $i \in \{1, \ldots, n \}.$

\end{theorem}


Finally, we apply our classification results to the linear plumbings that admit a concave contact boundary. The statement does not involve any contact toric geometry, but the proof relies on it.

\begin{corollary} \label{cor}
Up to contactomorphism , there is unique tight and two overtwisted contact structures on any lens space $L(k,l)$ that can be realised as a concave contact boundary of a linear plumbing over spheres.
The tight one is the unique universally tight contact structure and  overtwisted contact structures are obtained by performing the half-Lutz twist and a full Lutz-twist to the tight one.
\end{corollary}
  
 
 \subsection*{Acknowledgments}
 The authors are grateful to Miguel Abreu and Klaus Niederkrüger for helpful conversations and the organizers of the Workshop on Symplectic Topology at the University of Belgrade in 2024. AM was partially supported by the Ministry of Education, Science and Technological Development, Republic of Serbia, through the project 451-03-66/2024-03/200104. LS was supported by NSF DMS 2042345 and a Sloan Fellowship.

\section{Preliminaries} \label{sec:intro}

A contact structure $\xi$ on a manifold $Y^{2n+1}$ is  a codimension 1 distribution  on $TY$ that is locally given as the kernel of a differentiable 1-form $\alpha$ where
$\alpha\wedge d\alpha^n$ nowhere vanishes. If such a 1-form $\alpha$ is globally defined, then it is called a contact form and 
a contact structure $\xi=\ker\alpha$  is  coorientable. 

In this article we address the dichotomy of overtwisted vs. tight contact structures on a 3-manifold.
 A contact structure on a 3-manifold is called overtwisted if it contains an overtwisted disc, i.e. an embedded disc that is tangent to the contact structure along the boundary. If it is not overtwisted, then a contact structure is called tight. According to Eliashberg (\cite{Eliashberg}), the classification of overtwisted contact structures on closed 3-manifolds is a purely topological, namely,  in each homotopy class of tangent 2-plane fields there is a unique overtwisted contact structure, up to isotopy. 
On the other hand,  tight contact structures on a manifold are in general more rigid and harder to classify. As shown by Honda in \cite{Honda}, the number of tight contact structures on a lens space $L(k,l)$ varies, depending on $k$ and $l$. However, there is unique universally tight contact structure on every lens space $L(k,l)$ and this is the contact structure obtained by quotienting the unique tight contact structure on $S^3.$
In Section \ref{sec:tight} we show that this is the unique tight contact structure on any lens space $L(k,l)$ that admits a toric action.

\subsection{Contact toric manifolds}  \label{sec:classification}
A  contact manifold $(Y^{2n-1},\xi)$  equipped with an effective  $T^n=(\mathbb R/ \mathbb Z)^n $ action that preserves the contact structure is called a \emph{contact toric manifold}. 
In general, a toric action may not preserve every contact form of an invariant contact toric structure $\xi$. However, if $\alpha$ is any contact form for $\xi,$ then
$\alpha_{inv}=\int_{\theta\in T^n}(\theta^*\alpha )d\theta$
is an invariant contact form.
Therefore, from now on, we always assume to work with an invariant contact form $\alpha. $

To every contact toric manifold $(Y^{2n-1}, \xi)$ with an invariant contact form $\alpha$ we associate $\alpha$-moment map  $H_{\alpha }=(H_1,\ldots, H_n): Y\to\mathbb R^{n}$  uniquely
defined by 
$$H_k=\alpha(X_k), \hskip2mm k=1,\ldots,n,$$
where $X_k$, $k=1,\ldots, n$ are infinitesimal generators of the toric action.
Moreover, the natural lift of the toric action on $(Y^{2n-1}, \xi)$  to the symplectization $(Y \times  \mathbb R_t, d(e^t\alpha))$ is  a toric action and therefore the symplectization is a symplectic toric manifold with a moment map  $H_{\xi}=e^tH_{\alpha}.$ The moment cone of a contact toric manifold is defined as a moment map image of the symplectization together with the origin. 
While the $\alpha$-moment map depends on the choice of an invariant contact form, the moment cone depends only on the contact structure. By performing an automorphism of the torus $T^n$ that acts on a contact manifold, the corresponding moment cone changes by an $SL(n,\mathbb Z)$ transformation. If there exists a diffeomorphism between two contact toric manifolds that preserves the contact structures as well as the toric actions  we say that these contact toric manifolds are  \emph{equivariantly contactomorphic}.

\begin{example}\label{ex:sphere}
The standard contact sphere $(S^{2n-1}, \ker(\alpha_{st}=\frac{i}{4}\sum_{k=1}^n(z_kd\bar{z}_k-\bar z_kdz_k))))$ equipped with the action
$$(e^{i2\pi t_1},\ldots,e^{in\pi t_n})*(z_1,\ldots ,z_n)\mapsto(e^{i2\pi t_1}z_1,\ldots ,e^{i2\pi  t_n}z_n)$$
is a contact toric manifold. The infinitesimal generators of the action are vector fields $X_k=2\pi i(z_k\frac{\partial}{\partial z_k}-\bar{z}_k\frac{\partial}{\partial \bar{z}_k}), k=1,\ldots,n$.
Therefore, 
\begin{equation}\label{standard H}
H_{\alpha_{st}}(z_1,\ldots, z_n)=\pi(|z_1|^2,\ldots,|z_n|^2)
\end{equation}
and the $\alpha_{st} $-moment map image is the standard $n$-simplex. The corresponding moment cone is equal to $\mathbb R^n_{\geq 0 }.$

\end{example}

If the toric action is free then any contact toric 3-manifold is equivariantly contactomorphic to
$(T^3,  \xi_n= \ker(\cos(nt)\,d\theta_1+\sin (nt)\, d\theta_2) ),$ for some $n  \in  \mathbb N$ (\cite[Theorem 2.18. (1)]{Lerman2}). For differing values of $n$, these contact structures were shown to be non-contactomorphic by Kanda~\cite{Kanda}, so the classification in the free case was completed by Lerman. Note that the corresponding moment cone is always $\mathbb R^2$, for all values of $n$.
If the toric action is non-free, then there is not a complete classification of which contact structures admit such an action up to contactomorphism. However, Lerman did show that all such contact manifolds fall in a certain class which we will now describe. In Section \ref{sec:tight} and Section \ref{sec:overtwisted} we will classify these contact structures explicitly and describe them through contact topological constructions.

\begin{definition} \label{def:Yt1t2}
Given two real numbers $0\leq t_1<t_2$ such that  $(\cos t_i,\sin t_i)$ is proportionate to $(m_i,n_i)\in \Z^2$, $i=1,2,$ we define a contact toric manifold $(Y(t_1,t_2),\xi_{t_1,t_2})$ as follows. Start from $T^2\times [0,1]$ with coordinates $(\theta_1,\theta_2,t)$, the contact structure $\ker(\cos (t_1(1-t)+t_2t)d\theta_1+\sin (t_1(1-t)+t_2t) d\theta_2)$, and the toric action given by the standard rotation of $T^2$ coordinates.
Collapse the  tori $T^2\times \{0\}$ and $T^2\times \{1\}$  along the circles of slopes $(-n_1,m_1)$ and $(n_2,-m_2)$ respectively.
 The quotient space inherits a contact toric structure and we denote this manifold  by $(Y(t_1,t_2),\xi_{t_1,t_2})$.
\end{definition}

The moment cone of $(Y(t_1,t_2),\xi_{t_1,t_2})$ is the union of the rays from the origin of angle $\theta$ for $t_1\leq \theta \leq t_2$. If $t_2-t_1\geq 2\pi$, this will be all of $\R^2$. If $t_2-t_1<2\pi$, we will see a cone bounded by a ray at angle $t_1$ and a ray at angle $t_2$ (see Figure~\ref{fig:rays}). The two circle orbits of $(Y(t_1,t_2),\xi_{t_1,t_2})$ have moment images on the rays at angle $t_1$ and $t_2$. The slopes which are collapsed in $T^2\times \{0\}$ and $T^2\times \{1\}$ giving rise to these circle orbits are the normal vectors to the rays at angle $t_1$ and $t_2$.
 
\begin{figure}
\centering
\includegraphics[width=12cm]{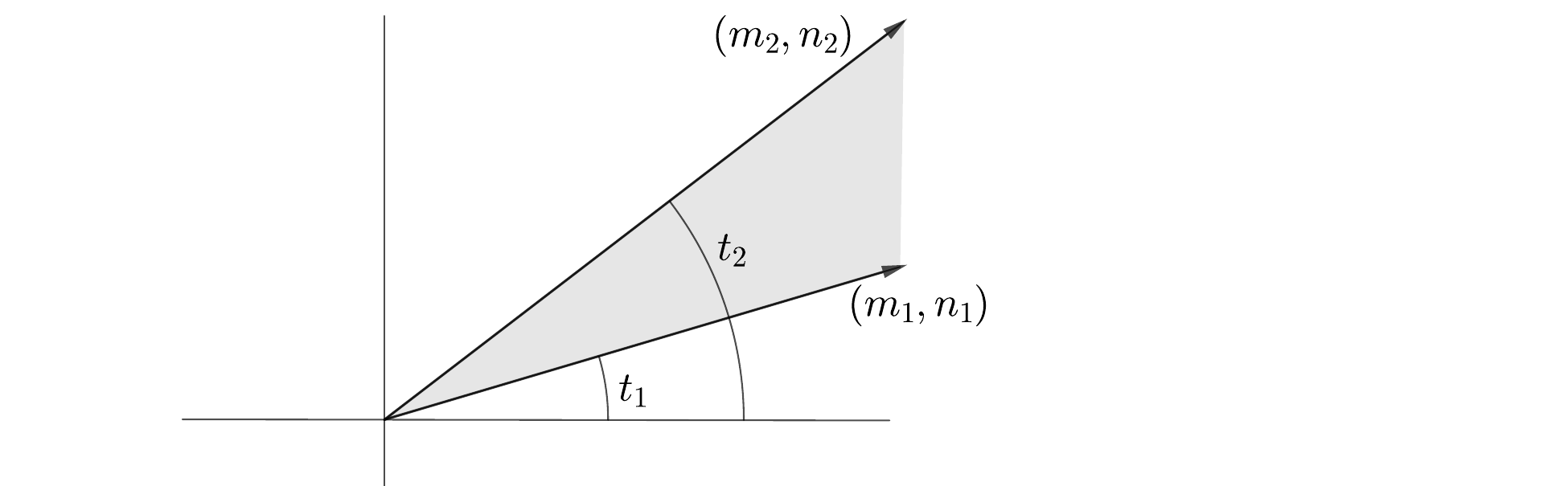}
\caption{A moment cone of $(Y(t_1,t_2),\xi_{t_1,t_2})$}
\label{fig:rays}
\end{figure}

\begin{theorem}(\cite[Theorem 2.18. (2)]{Lerman2})\label{theorem Lerman class}
Every compact connected contact toric manifold with a non-free toric action is equivariantly contactomorphic to  $(Y(t_1,t_2),\xi_{t_1,t_2})$ for some pair of real numbers $t_1,t_2$ with $0\leq t_1<2\pi$, $t_1<t_2$, such that  $(\cos t_i,\sin t_i)$ is proportionate to $(m_i,n_i)\in \Z^2$, $i=1,2,$
\end{theorem}

As $SL(2, \mathbb Z)$ transformations of moment cones preserve the corresponding contact toric structures, a given contact toric manifold can be realized by multiple pairs $t_1,t_2$.

\begin{remark}\label{rem:lens_space} By performing an $SL(2, \mathbb Z )$ transformation of the moment cone, we may assume $t_1=0$. If  we denote by $(l,k)$ the slope of the second ray pointing out of the origin (namely $\tan(t_2)=k/l$) then the corresponding contact toric manifold is diffeomorphic to a lens space $L(k,l).$ To see this, notice that a lens space $L(k,l)$ can be obtained from
 $T^2\times [0,1]$ by 
collapsing the  tori $T^2\times \{0\}$ and $T^2\times \{1\}$  along the circles of slopes $(0,1)$ and $(k,-l)$. 
Further, the classical theorem by Reidmeister  states that $L(k,l)$ is diffeomorphic to $L(k',l')$ if and only if
$ k=k', l'\equiv \pm l^{\pm1} \hskip2mm (\textrm{mod})\hskip2mm  k $. That is,
 $L(k,l)$ is diffeomorphic to $L(k,-l),$ $L(k, l+nk),$ $L(k,-l+nk),$ $L(k,r+nk)$ and $L(k,-r+nk),$ for any $n\in\mathbb Z$, where $lr-ks=1,$ for some $s\in \mathbb Z.$
Thus, we are able to collect all the moment cones that correspond to $L(k,l).$ These are precisely all the cones whose first ray is $(1,0)$ and the second ray is 
$ (\pm(l+nk), \pm k )$ or $(\pm (r+nk),\pm k),$  for any $n\in\mathbb Z$, where $lr-ks=1,$ for some $s\in \mathbb Z.$
\end{remark}

Moreover, the numbers $t_1$ and $t_2$ are essential in determining when the corresponding contact structure is tight or overtwisted.

 \begin{theorem} \cite[Theorem 3.2.]{MNRSTW} \label{thm:tight-ot}
 $(Y(t_1,t_2),\xi_{t_1,t_2})$
 is overtwisted if and only if  $t_2-t_1> \pi$.
\end{theorem} 

This property will be essential in the proof of Theorem  \ref{thm:tight} and Theorem  \ref{thm:overtwisted}.

\subsection{Non-free contact toric 3-manifold as a concave boundary of a linear plumbing}\label{sec:mnrstw} We briefly recall the construction from \cite{MNRSTW} as it will be used to prove Theorem \ref{thm:toric_is_plumbing}.

We denote by $(s_1, \ldots, s_n)$ the linear plumbing of disk bundles over spheres with self-intersection numbers $s_1, \ldots, s_n.$ If  $s_i \geq0,$ for at least one index $i \in \{1, \ldots, n \}$ then such a plumbing can be equipped with a global symplectic toric structure with concave contact boundary, such that the spheres in the base of the plumbing correspond to the edges in the moment map image~\cite[Theorem 4.1.]{MNRSTW}. For the details of the construction we refer to \cite[Section 4]{MNRSTW}. The moment map image of the plumbing $(s_1,s_2)$, where $s_1 \geq0$ is shown in Figure  \ref{fig:l-shape} and it is often called an L-shape. The inner curve corresponds to the boundary of the plumbing and the two points on the rays correspond to singular toric orbits. Except end points of edges that correspond to fixed points of the toric action, all other points on the rays in the moment image correspond to circle orbits.
 
    \begin{figure}
\centering
\includegraphics[width=11cm]{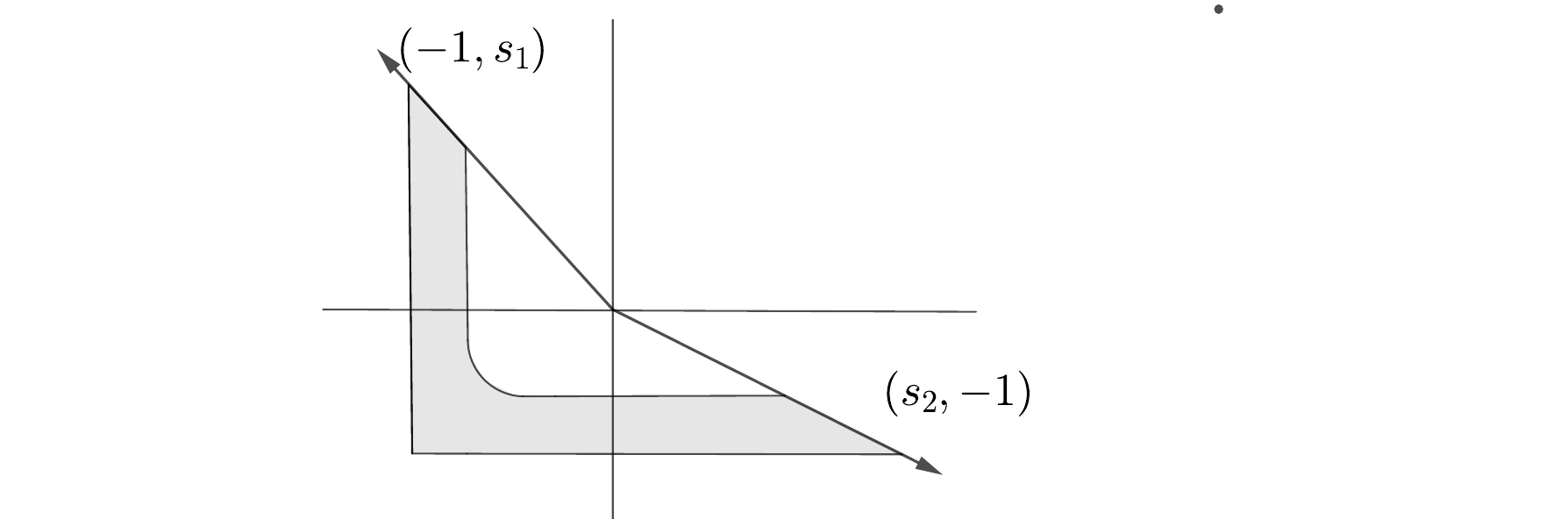}
\caption{L-shape- the moment map image of $(s_1,s_2)$}
\label{fig:l-shape}
\end{figure}      

The moment map image of the plumbing $(s_1, \ldots, s_n)$, where $s_i\geq 0$, is obtained by gluing the moment map images of the plumbings 
$$ (s_1,0),\ldots, (s_{i-1},0),(s_i,s_{i+1}),(0,s_{i+2}),\ldots, (0,s_n)$$
via transformations 
$A_j=\begin{bmatrix}
 -s_j & -1 \\
1 &  0
\end{bmatrix},$ for all $j=2,  \ldots, n-1,$ in the following way. The L-shape of $(0,s_n)$ is glued via $A_{n-1}$ to the L-shape of $(0,s_{n-1})$ (Figure \ref{fig:gluing}), where the points on the rays $(s_{n-1},-1)$ and $(-1,0)$ correspond to regular torus orbits, except for the points on edges that correspond to circles. In this way, the pre-image of the points on one ray is a solid torus. Next, the obtained moment map image is glued via $A_{n-2}$ to the L-shape of $(0,s_{n-2})$ and we continue until we glue all L-shapes to the L-shape $(s_1,0).$
In this total moment map image only the point on the first ray of $(s_1,0)$ and the point on the second ray of $(0,s_n)$ correspond to singular orbits as in the case of one L-shape. All points on the inner rays  correspond to regular (full-dimensional torus) orbits, except for the points on the interior of edges that correspond to circle orbits. Note that in order for these gluing maps to align the boundaries, the coordinates of the vertex in each L-shape must be chosen carefully. In~\cite[Theorem 4.1]{MNRSTW}, we give an explicit way to choose these coordinates consistently in the third quadrant. This relies on the assumption that in each L-shape $(a,b)$, at least one of $a$ and $b$ is non-negative.
    \begin{figure}
\centering
\includegraphics[width=14cm]{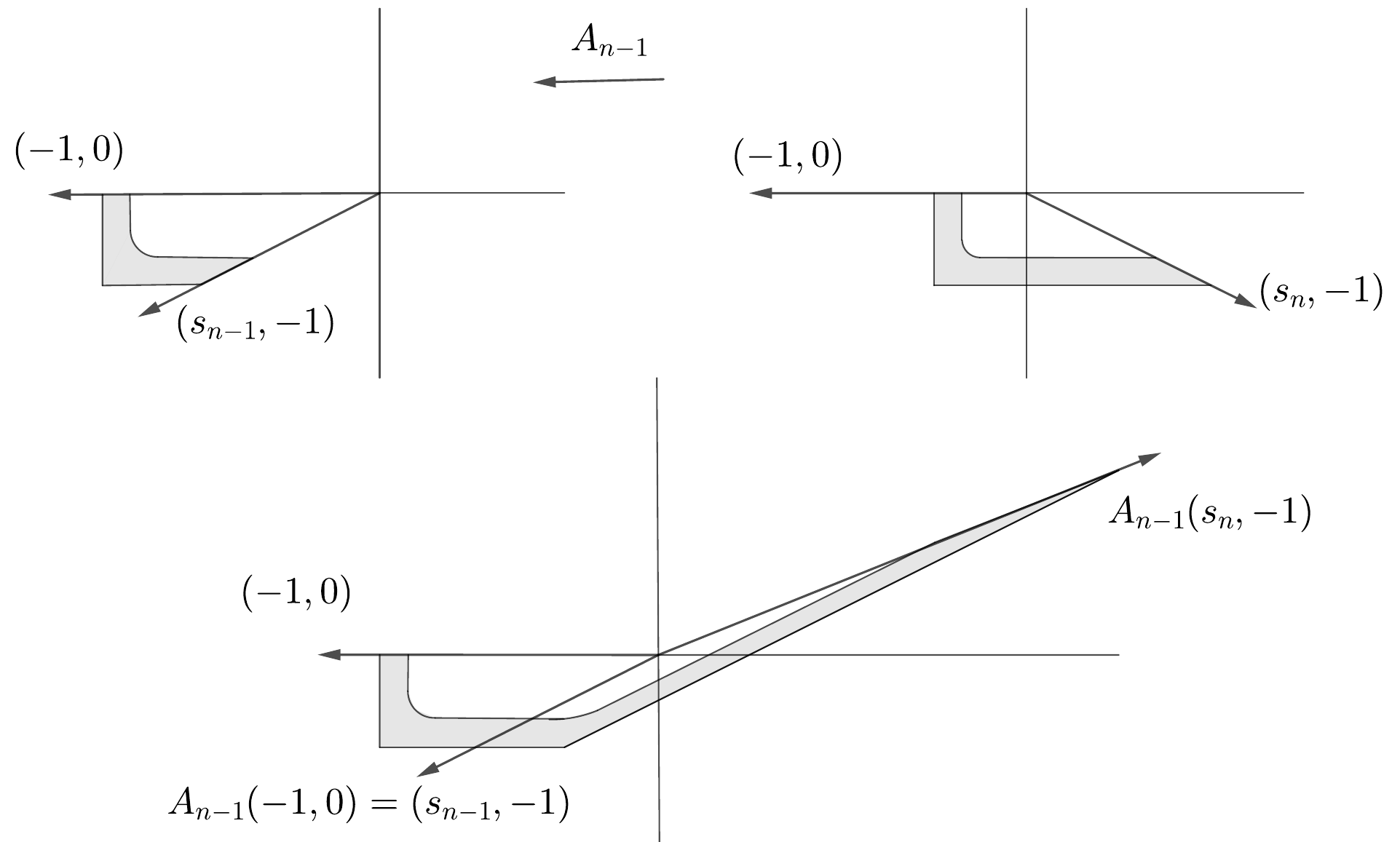}
\caption{Gluing of $(0,s_n)$ to $(0,s_{n-1})$}
\label{fig:gluing}
\end{figure}

Moreover, the plumbing $(s_1, \ldots, s_n)$, where $s_i \geq0,$ for at least one index $i \in \{1, \ldots, n \}$, admits a Liouville vector field defined near the boundary and pointing toward interior. This Liouville vector field is  invariant under the toric action and, therefore, the boundary admits a concave contact toric structure. The rays of the corresponding moment cone (pointing out of the origin) are given by

\begin{equation}\label{eq:rays}
R_1=(-1,s_1)\hskip2mm \textrm{and} \hskip2mm R_2=\begin{cases}
        (s_2,-1),&\mbox{ if $n=2$,  }\\  
  A_2  \cdots A_{n-1}(s_n,-1),&\mbox{ if $n\geq3,$}
 \end{cases}
\end{equation}
where 
$A_j=\begin{bmatrix}
 -s_j & -1 \\
1 &  0
\end{bmatrix},$ for all $j=2,  \ldots, n-1.$

\section{Tight contact toric structures}\label{sec:tight}

In this section we focus on contact toric 3-manifolds with tight contact structures. 

If the action is free then, as mentioned in Section \ref{sec:classification} any such manifold is equivariantly contactomorphic to
$(T^3,  \xi_n= \ker(\cos(nt)\,d\theta_1+\sin(nt)\, d\theta_2) ),$ for some $n  \in  \mathbb N.$ 

If the action is non-free, then, according to Theorem  \ref{thm:tight-ot}  the corresponding moment cone spans an angle $ \leq  \pi.$
In fact,  there is a bijection between such moment cones  (up to $SL(2,\mathbb Z )$-transformations) and tight contact manifolds with a non-free toric action (up to equivariant contactomorphisms).
We now describe this bijection explicitly. We suppose $t_1=0$ in all cases.

 \begin{itemize}

 \item The convex cone spanned by the rays $(1,0)$ and $(0,1)$ corresponds to 
$(L(1,0)=S^3, \ker(\alpha_{st}=\frac{i}{4}\sum_{k=1}^2(z_kd\bar{z}_k-\bar z_kdz_k))$ with the standard toric action
$$(e^{i2\pi t_1},e^{i2\pi t_2})*(z_1,z_2)\mapsto(e^{i2\pi t_1}z_1,e^{i2\pi  t_2}z_2).$$

 \item The convex cone spanned by the rays $(1,0)$ and $(l,k)$, for any $k \in  \mathbb N$ and any $l  \in  \mathbb Z$ corresponds to
 a lens space $L(k,l)$ with the following contact toric structure. Identify $L(k,l)$ with a quotient space of the unit sphere $S^3$ by the free $\mathbb Z_k$ action 
$$e^{i 2\pi /k}\ast(z_1,z_2)\to (e^{i 2\pi /k}z_1, e^{i 2\pi l/k}z_2).$$
As the contact form $\alpha_{st}$ on $S^3$ is invariant under this action,  it induces a contact form $\alpha_{kl}$ on $L(k,l).$ 
By reparametrising  the standard toric action  on $(S^3,\ker\alpha_{st})$ and dividing the first circle acting by $k$ we obtain a well defined toric action on $(L(k,l), \ker\alpha_{kl})$

$$(e^{i2\pi t_1},e^{i2\pi t_2})*(z_1,z_2)\mapsto(e^{i2\pi t_1/k}z_1,e^{i2\pi  l t_1/k}e^{i2\pi  t_2}z_2).$$

The corresponding  moment map is given by
$$H([z_1,z_2])=\pi(\frac{1}{k}(|z_1|^2+l |z_2|^2), |z_2|^2).$$
and the moment cone is spanned by the rays $(1,0)$ and $(l,k)$. See  Figure \ref{convex lens} on the left.
  
 \item The convex cone spanned by the rays $(1,0)$ and $(-1,0)$ corresponds to  $(S^1_{\theta}\times S^2_{h,z}, \ker (hd\theta+\frac{i}{4}(zd\bar{z}-\bar zdz)))$ with the toric action defined   by
 $$(e^{i2\pi t_1},e^{i2\pi t_2})*(e^{i2\pi \theta},z,h)\mapsto(e^{i2\pi (t_1+\theta)},e^{i2\pi   t_2}z,h).$$
 Namely,   the moment map is given by
 $$H(\theta, h,z)=(h,\pi|z|^2).$$
The corresponding moment cone is shown in  Figure \ref{convex lens} on the right. 
 
  \end{itemize} 
 
    \begin{figure}
\centering
\includegraphics[width=12cm]{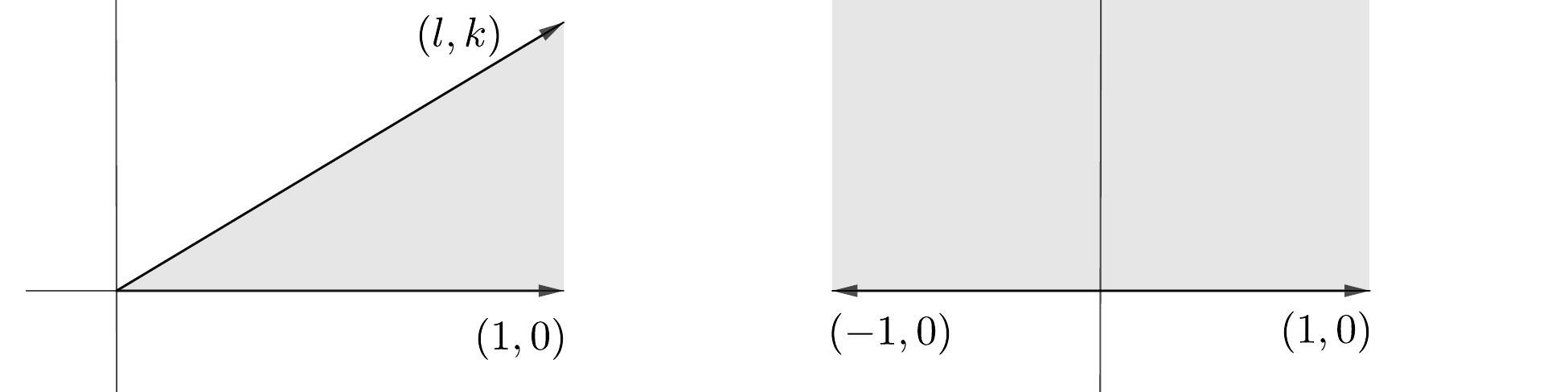}
\caption{Moment cone of a Lens space $L(k,l)$ with a standard tight contact structure, left $l>0$,  right $k=0.$ }
\label{convex lens}
\end{figure}

We are now ready to prove Theorem \ref{thm:tight}.

\begin{proof}[Proof of Theorem \ref{thm:tight}]
Suppose $(Y(t_1,t_2),\xi_{t_1,t_2})$ and $(Y(t_1',t_2'),\xi_{t_1',t_2'})$ are two tight contact toric 3-manifolds such that the underlying smooth 3-manifolds $Y(t_1,t_2)$ and $Y(t_1',t_2')$ are diffeomorphic. By Theorem~\ref{thm:overtwisted}, $t_2-t_1\leq \pi$. By Remark~\ref{rem:lens_space}, we assume $t_1=t_1'=0$, and $Y(t_1,t_2)=L(k,l)$ where $(\cos t_2, \sin t_2)$ is proportionate to $(l,k)$, $Y(t_1',t_2')=L(k',l')$ where $(\cos t'_2, \sin t'_2)$ is proportionate to $(l',k').$ Then $(Y(t_1,t_2),\xi_{t_1,t_2}) = (L(k,l),\ker(\alpha_{kl}))$ and $(Y(t_1',t_2'),\xi_{t_1',t_2'}) = (L(k',l'),\ker(\alpha_{k'l'}))$.

The classical theorem by Reidemeister says that 
 $L(k,l)$ is diffeomorphic to $L(k',l')$ if and only if
$  k=k', l'\equiv \pm l^{\pm1} \ (\textrm{mod})\  k .$ 
 Thus, $(k',l')$ must be one of
  $(k,-l),$ $(k, l+nk),$ $(k,-l+nk),$ $(k,r+nk)$, or $(k,-r+nk),$ for some $n\in\mathbb Z$ where $lr-ks=1,$ for some $s\in \mathbb Z.$
The moment cone of $(Y(t_1,t_2),\xi_{t_1,t_2})$ is the convex cone spanned by the rays $(1,0)$ and $(l,k)$, and we denote this by $C(k,l)$. Similarly, the moment cone of $(Y(t_1',t_2'),\xi_{t_1',t_2'})$ is the convex cone $C(k',l')$ spanned by the rays $(1,0)$ and $(l',k')$. We now relate these moment cones for the different possible cases of $(k',l')$.

The convex cone $C(k,l)$ is mapped by the transformations
$\begin{bmatrix}
1 & n \\
0 & 1 
\end{bmatrix}\in  SL(2, \mathbb Z) $ and $\begin{bmatrix}
r+nk & -s-nl \\
k & -l 
\end{bmatrix}\in -SL(2, \mathbb Z) $ to the convex cones $C(k,l+nk)$ and $C(k,r+nk)$, respectively, while 
the convex cone for $C(k,-l)$ is mapped by the transformations
$\begin{bmatrix}
1 & n \\
0 & 1 
\end{bmatrix}\in  SL(2, \mathbb Z) $ and $\begin{bmatrix}
-r+nk & -s+nl \\
k & l 
\end{bmatrix}\in -SL(2, \mathbb Z) $ to the convex cones $C(k,-l+nk)$ and $C(k,-r+nk).$ 
Therefore,  the corresponding contact structures  $\ker\alpha_{kl}$  are equivariantly contactomorphic. 

It is left to compare the contact structure $\ker\alpha_{kl}$ on $L(k,l)$ and the contact structure $\ker\alpha_{k,-l}$ on $L(k,-l).$ A diffeomorphism 
$$(z_1,z_2)\to(z_1, \bar z_2)$$
on $S^3$ induces a diffeomorphism from $L(k,l)$ to $L(k,-l)$ and the pull-back of this diffeomorphism maps  $\alpha_{k,-l}$ to  $\alpha_{k,l}.$

Therefore, all tight contact structures on a Lens space $L(k,l)$ that admit a toric action are contactomorphic.

\end{proof}

%
%
%
%

\section{Overtwisted contact toric structures}\label{sec:overtwisted}

We now focus on contact toric 3-manifolds with an overtwisted contact strucuture. Note that 
the toric action has to be non-free, because the contact toric 3-manifold with a free toric action is equivariantly contactomorphic to 
$(T^3,  \xi_n= \ker(\cos ntd\theta_1+\sin nt d\theta_2) ),$  for some $n  \in  \mathbb N,$ and  these contact structures are tight. 
Further, according to Theorem  \ref{thm:tight-ot},  the  moment cone corresponding to an overtwisted contact toric structure  spans an angle $ >  \pi. $ In contrast to the tight contact toric structures, in general, overtwisted contact toric structures are not classified by the corresponding moment cones. Namely, 
if  $t_2-t_1\geq2\pi$ then the moment cone  of $(Y_{t_1,t_2},\xi_{t_1,t_2})$ is  
$\mathbb R^2$. Thus,  non-diffeomorphic manifolds can have the same moment cone. Furthermore, according to Lerman (see \cite{Lerman1}), the overtwisted contact structures of $(Y_{t_1,t_2},\xi_{t_1,t_2})$ and 
$(Y_{t_1,t_2+2n\pi},\xi_{t_1,t_2+2n\pi})$
 are homotopic, for all $n\in\mathbb N.$ 
This rotation of the second ray can be also explained in terms of  the full Lutz twist.
      We first recall the relevant definitions. For more details we refer to  \cite{Geiges}.
      
Let $L$ be a transversal knot of the contact structure $(Y, \xi)$. Then, there is a  small neighborhood of $L$ contactomorphic to $(S_{\theta}^1\times D^2_{t,r<\varepsilon},\xi= \ker(d\theta+r^2dt)$, where $L$ is identified with $S^1\times\{0\}$.  The following $T^2$ action on this neighborhood

\begin{equation}\label{eq:action}
(e^{i2\pi t_1},e^{i2\pi t_2})*(e^{i2\pi\theta},re^{i2\pi t})\mapsto(e^{i2\pi (t_1+\theta)},re^{i2\pi  (t_2+t)})
\end{equation}
preserves $\xi$ and, thus, it is a toric action. The corresponding moment map is given by
$$H(\theta, t,r)=(1,r^2)$$
and the moment map image is shown on the left in Figure \ref{fig:lutz}.

  \begin{definition}  
Replace the contact structure $\xi$ on $S^1\times D^2_{t,r<\varepsilon}$ with a contact structure
  $\xi' =\ker(h_1(r)d\theta+h_2(r)dt)$ where  $h_1 (r), h_2 (r): [0, \varepsilon) \to \mathbb R$ are smooth functions that satisfy the contact condition  $h_1(r)h_2' (r)- h_2 (r)h_1'(r)  \neq0$,  for all $ 0\leq r< \varepsilon$,    and $\xi'$ coincides with $\xi$ outside $S^1 \times D^2_{t,r<\varepsilon}$.
This procedure  is called:
 \begin{itemize}
 \item[a)] a  half-Lutz twist if 

 $$h_1(r)=-1, h_2(r)=-r^2,\hskip2mm\textrm{for} \hskip2mm r\in [0,\varepsilon/3],$$

  $$h_1(r)=1, h_2(r)=r^2, \hskip2mm\textrm{for} \hskip2mm r\in  [2\varepsilon/3,\varepsilon);$$

and the winding number of the curve $(h_1(r),h_2(r)),$ $r\in [0,\varepsilon]$ is $0$.

\item[b)] a full-Lutz twist if

$$h_1(r)=1, h_2(r)=r^2 \hskip2mm\textrm{for} \hskip2mm r\in [0,\varepsilon/3]\cup [2\varepsilon/3, \varepsilon)$$
and the winding number of the curve $(h_1(r),h_2(r))$, $r\in[0,\epsilon]$, is 1. 
\end{itemize}

\vskip2mm

Note that the $T^2$ action (\ref{eq:action}) preserves the contact structure $\xi' $ and, therefore, it is a toric action. The corresponding moment map is given by
$$H(\theta, t,r)=(h_1(r),h_2(r))$$
      and the moment map images are shown on the right in Figure \ref{fig:lutz}.
      
           \end{definition}

         \begin{figure}
\centering
\includegraphics[width=11cm]{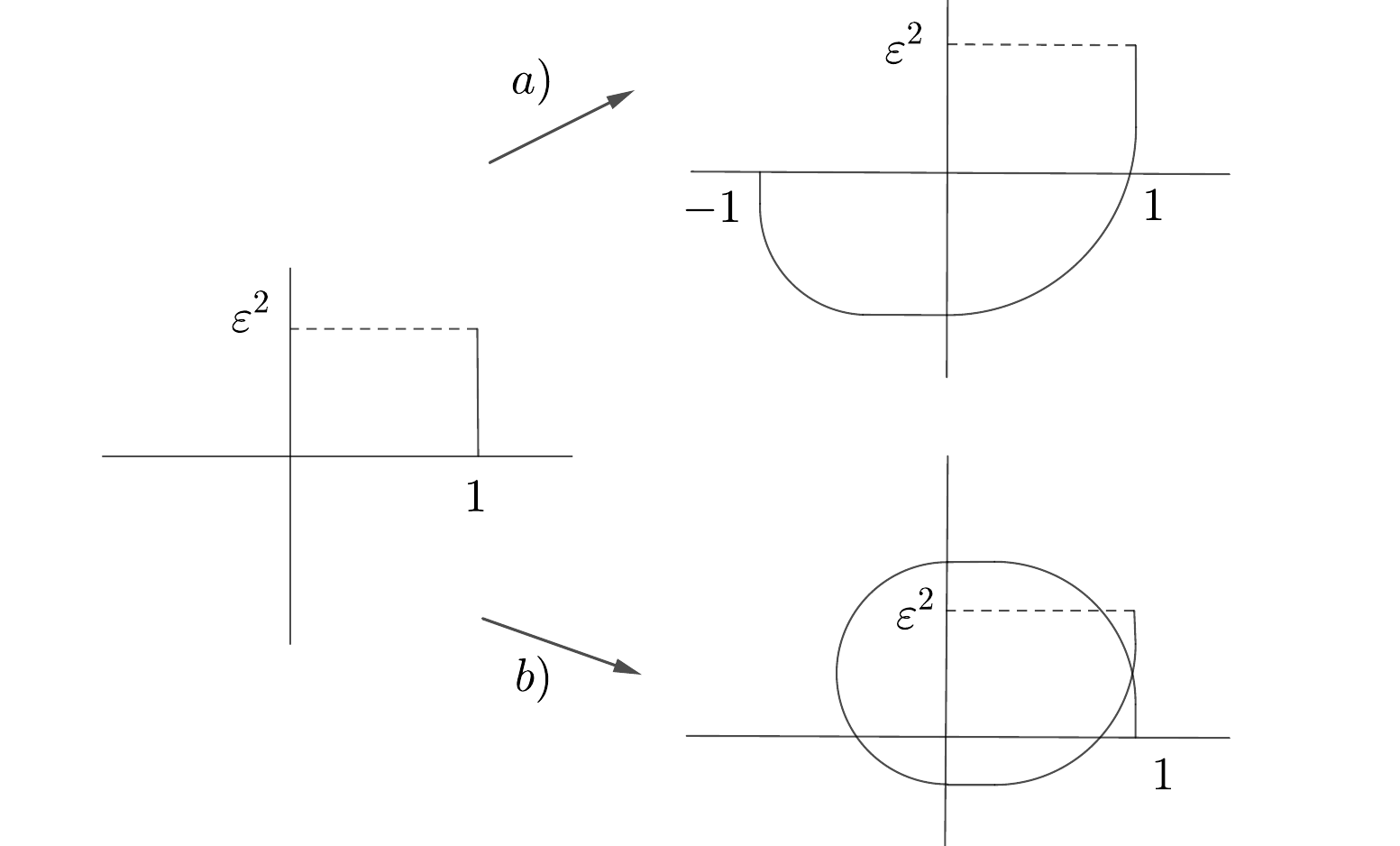}
\caption{a) A half-Lutz twist,  b) a full-Lutz twist }
\label{fig:lutz}
\end{figure}      

We now make use of the Lutz twists equipped with a toric action.
       Note that for a contact toric manifold $(Y_{t_1,t_2},\xi_{t_1,t_2})$ the circle orbits $L_1$ and $L_2$ corresponding to the numbers $t_1$ and $t_2$ are transversal to the contact structure $\xi,$
     since the Hamiltonian function $H$ does not vanish along them. Therefore, we may perform Lutz twists along either of them.

      \begin{proposition}\label{prop: Lutz}  Let $L$ be a circle orbit corresponding to $t_1$ or $t_2$
      in $(Y_{t_1,t_2},\xi_{t_1,t_2})$.      \begin{itemize} 
\item[a)] The half-Lutz twist of $(Y_{t_1,t_2},\xi_{t_1,t_2})$  along $L$ is equivariantly contactomorphic to 
$(Y_{t_1,t_2+\pi},\xi_{t_1,t_2+\pi}).$

  \item[ b)] The full-Lutz twist of  $(Y_{t_1,t_2},\xi_{t_1,t_2})$ along $L$ is equivariantly contactomorphic to  $(Y_{t_1,t_2+2\pi},\xi_{t_1,t_2+2\pi}).$
      \end{itemize}      
      \end{proposition}          

      \begin{proof} a) Without loss of generality assume $t_1=0.$
      We perform a half-Lutz twist along the orbit $L$ that corresponds to $t_1=0$ in the following way. 
      
      Let us first show that there is a small neighborhood of $L$ that is equivariantly contactomorphic to 
      $(S_{\theta}^1\times D^2_{t,r<\varepsilon}, \ker(d\theta+r^2dt))$   with the toric action precisely given by
   (\ref{eq:action}).

   According to Lerman, there is a small neighborhood $u$ of $L$ that is equivariantly contactomorphic to
$$(T^2\times [0,\varepsilon'), \ker(\cos (t_2s)d\theta_1+\sin (t_2s) d\theta_2) )/_{\sim}, $$

where  $T^2 \times \{ 0\}$ is collapsed along the circle of slope $(0,1)$ and this quotient corresponds to $L.$ Therefore, this neighborhood of $L$ is diffeomorphic to $S^1 \times D^2.$

   Denote by $\alpha_0'$ the contact form on $T^2 \times [0, \varepsilon')/ _{\sim} $ where
   $\pi^*\alpha_0'=\cos (t_2s)d\theta_1+\sin (t_2s) d\theta_2$
 and
    $\pi: T^2 \times  [0, \varepsilon')   \to T^2 \times [0, \varepsilon')/_{  \sim}  $ is the natural projection. Next, if $\varphi:  S_{\theta}^1\times D^2_{t,r<\varepsilon'}  \to  T^2 \times [0, \varepsilon')/_{ \sim}$ is a diffeomorphism given by
    $\varphi(\theta,t,r)=(\theta_1=\theta, \theta_2=t, s=r)$
     we  denote  
             $ \alpha_0''= \varphi^{*}  \alpha_0'.$
                  Then
             $ \alpha_0''= \cos (t_2r)d\theta+\sin (t_2r) dt.$
             In the same contact structure, we consider the contact form 
               $$ \alpha_0= d\theta+\tan (t_2r) dt.$$             Denote
             $$\alpha_1=d\theta+r^2dt.$$
             Then, the map $\Phi:S^1\times D^2_{r<\varepsilon'}\to S^1\times D^2_{r<\varepsilon}$, where $\varepsilon=\sqrt{\tan(t_2 \varepsilon')}$, defined by $\Phi(\theta,t,r)=(\theta, t, \sqrt{\tan(t_2 r)})$
  is a diffeomorphism that satisfies $\Phi^*\alpha_1=\alpha_0$ . As $ \Phi$ is also invariant under the toric action (\ref{eq:action}) we conclude that              
             $\alpha_0$ and $\alpha_1$ are equivariantly contactomorphic. Therefore, 
             $(T^2 \times [0, \varepsilon')/_{  \sim} ,\ker\alpha_0' )$ is equivariantly contactomorphic to  $(S_{\theta}^1\times D^2_{t,r<\varepsilon}, \ker(d\theta+r^2dt))$ with respect to the toric action (\ref{eq:action}).

   We now perform the half-Lutz twist as described above. The   action (\ref{eq:action}) on the Lutz twisted neighborhood
       $(S_{\theta}^1\times D^2_{t,r<\varepsilon},  \ker(h_1(r)d\theta+h_2(r)dt))$ preserves the contact structure, and, thus, it is a toric action. The moment map is given by $H(\theta, t,r)=(h_1(r),h_2(r)).$ Therefore, the new contact structure is also a toric contact structure. On the complement of $S_{\theta}^1\times D^2_{t,r<\varepsilon}$ the contact toric form is not changed. Thus, it follows that the new toric contact manifold is $(Y(t_1-\pi,t_2,\xi_{t_1-\pi,t_2}).$ By performing the $\pi$-rotation, we conclude that the new toric contact manifold is equivariantly contactomorphic to $(Y(t_1,t_2+\pi),\xi_{t_1,t_2+\pi})$.
       
      One obtains an analogous result by performing the half-Lutz twist along the other circle orbit.

      b) The proof can be derived similarly as in part a). However, we can also derive it in the following way. By performing the full Lutz twist on the tight contact structure  $\xi$ we obtain an overtwisted contact structure that is in the same homotopy class as 2-plane fields with $\xi$, see \cite[Lemma 3.17]{Geiges}. On the other hand,  by rotating the second ray by $2\pi$ we obtain an overtwisted contact structure also in the same homotopy class as 2-plane fields, see \cite[Theorem 3.2]{Lerman1}. Finally, according to Eliashberg (\cite{Eliashberg}), in every homotopy class of 2-plane fields on a 3-manifold there is unique, up to isotopy,  overtwisted contact structure and, therefore, these two overtwisted contact structures are isotopic.

      \end{proof}

 \begin{proof}[Proof of Theorem \ref{thm:overtwisted}.]

  We divide the proof into several steps.

  Fix an underlying diffeomorphism type $L(k,l)$ of a contact toric manifold.
   
    \textbf{Step 1.} \emph{  Collecting all the cones corresponding to overtwisted contact structures.}

    We consider all possible contact toric structures on the manifold $L(k,l)$. Note that each of these has the form $(Y({t_1,t_2}),\xi_{t_1,t_2})$. Without loss of generality, we will assume $t_1=0$.
   
   $\bullet$ Consider all $(Y({t_1,t_2}),\xi_{t_1,t_2})$ where $Y({t_1,t_2})\cong L(k,l)$ and $2\pi n< t_2-t_1 \leq 2\pi n+\pi$ for some $n\in \mathbb{N}$. Let $t_2' = t_2-2\pi n$, $t_1'=t_1$. Then $(Y({t_1',t_2'}),\xi_{t_1',t_2'})$ has a convex moment cone, $Y({t_1',t_2'})$ is diffeomorphic to $Y({t_1,t_2})$, and $\xi_{t_1',t_2'}$ is a tight contact structure. By Theorem~\ref{thm:tight}, any two tight contact toric structures on $L(k,l)$ are contactomorphic. As explained above, Lerman showed that $\xi_{t_1,t_2}$ and $\xi_{t_1',t_2'}$ are homotopic as 2-plane fields. By Eliashberg, any two overtwisted contact structures which are homotopic are contact isotopic, and in particular contactomorphic. Thus all contact toric manifolds $(Y({t_1,t_2}),\xi_{t_1,t_2})$ such that $Y({t_1,t_2})$ is diffeomorphic to $L(k,l)$ and $2\pi n< t_2-t_1 \leq 2\pi n+\pi$ for some $n\in \mathbb N$ are contactomorphic. Additionally, such overtwisted contact toric structures exist on each $L(k,l)$, since we can start with the a representative of the tight contact toric structure on $L(k,l)$ and add $2\pi$ to $t_2$. We will denote this overtwisted contact structure on $L(k,l)$ by $\xi_1$.
   

  $\bullet$   Next, consider $(Y({t_1,t_2}),\xi_{t_1,t_2})$ with $Y({t_1,t_2})\cong L(k,l)$ and $2\pi n+\pi < t_2-t_1 <2\pi(n+1)$ for $n\in \mathbb Z_{\geq 0}$. Set $t_2'=t_2-2\pi n$ and $t_1'=t_1$. Then $Y({t_1',t_2'})\cong Y({t_1,t_2})\cong L(k,l)$, and $\xi_{t_1,t_2}$ and $\xi_{t_1',t_2'}$ are homotopic and are both overtwisted. Thus $(Y({t_1,t_2}),\xi_{t_1,t_2})$ is contactomorphic to $(Y({t_1',t_2'}),\xi_{t_1',t_2'})$. Now let $t_2''=t_2'-\pi$ and $t_1''=t_1'$. Then $(Y({t_1'',t_2''}),\xi_{t_1'',t_2''})$ is the unique tight contact toric structure $(L(k,l),\ker(\alpha_{kl}))$ since $0<t_2''-t_1''< \pi$. By Proposition~\ref{prop: Lutz}, $(Y({t_1',t_2'}),\xi_{t_1',t_2'})$ is obtained from $(Y({t_1'',t_2''}),\xi_{t_1'',t_2''})$ by a half-Lutz twist along the transversal circle orbit $L$. Thus, any overtwisted contact structure on $L(k,l)$ with $2\pi n+\pi < t_2-t_1 <2\pi(n+1)$ is contactomorphic to the one obtained from $(L(k,l),\ker(\alpha_{kl}))$ by a half-Lutz twist along the transversal circle orbit of the toric action. We denote this contact structure on $L(k,l)$ by $\xi_2$.
 

Next, we will compare these two overtwisted contact structures $\xi_1$ and $\xi_2$ on $L(k,l)$ and show that they are not contactomorphic.

 \textbf{Step 2.} \emph{ Obstruction class $d_2(\xi_1,\xi_2).$ }
 
It is enough to show that the overtwisted contact structures $\xi_1$ and $\xi_2$ belong to different homotopy classes as oriented plane fields on $L(k,l)$. This can  be detected by  certain obstruction classes.

We briefly define the class
$$d_2(\xi_1, \xi_2) \in H^2(Y;\pi_2(S^2))= H^2(Y;\mathbb Z)$$
on a 3-manifold $Y$
that measures if  $\xi_1$ and $\xi_2$ are homotopic as 2-plane fields over the 2–skeleton of $Y$. To see this obstruction, recall that the homotopy classes of oriented plane fields on a 3-manifold are in 1-1 correspondence with the
homotopy classes of maps $f: Y\to S^2$. In turn, the homotopy classes of maps $f:Y  \to S^2$ are in 1-1 correspondence with  cobordism classes of framed (and oriented) links in $Y$ (these are called the corresponding Pontryagin manifolds, for the details we refer to \cite[Section 7]{Milnor}).
Then, associate to $\xi_1 $ and $\xi_2$ corresponding classes of links $[L_1]$ and $[L_2]$ and define
$$d_2(\xi_1,\xi_2 )=PD([L_1])-PD([L_2]),$$
where $PD$ denotes the Poincar\'e dual.
Then, $\xi_1$ and $\xi_2$ are homotopic as 2-plane fields over the 2–skeleton of $Y$ if and only if $d_2(\xi_1,\xi_2 )=0$  (\cite[Lemma 4.2.5.]{Geiges}).
 The following proposition of Geiges explains the relative $d_2$ obstruction between two contact structures related by a half-Lutz twist.

\begin{proposition} \cite[Proposition 4.3.3.]{Geiges}
Let $L$ be a transversal knot of the contact structure $(Y, \xi_1)$. If the contact structure $\xi_2$ is obtained from $\xi_1 $ by performing a half-Lutz twist along $L$ then
$$d_2(\xi_1, \xi_2) = PD([L]),$$
where $PD$ is the Poincar\'e dual of the first homology class represented by $L.$

\end{proposition}

Note that,  according to Proposition \ref{prop: Lutz}, the contact structure  $\xi_1$ is obtained from $\xi_2$ by performing a half-Lutz twist along a transversal orbit $L.$ (Moreover, $\xi_2$ is also obtained from $\xi_1$ by performing a half-Lutz twist.) This orbit is precisely the circle orbit obtained by collapsing 
 $T^2\times \{0\}$ along the circle of slope $(0,1)$ (as we suppose $t_1=0$ for both contact structures). 

Let us show that $d_2(\xi_1, \xi_2)$ is non-vanishing in the case of $ S^1 \times S^2$ and all lens spaces $L(k,l)$ different from $S^3.$
 
  If $Y=S^1 \times S^2 $ the chosen transversal orbit $L$ is precisely $S^1 \times  \{N \},$ where $N$ is the north pole of $S^2.$ It corresponds to the generator in 
  $H_1(S^1 \times S^2, \mathbb Z) \cong \mathbb Z$ and, therefore its Poincar\'e dual is non-vanishing.
 
  If $Y=L(k,l)$, we start  from the observation that $L(k,l)$ is obtained from $T^2  \times [0,1]$ by collapsing $T^2\times \{0\}$ and $T^2\times \{1\}$ along the circles of slopes $(0,1)$ and $(k,-l),$
  respectively. That is, in the definition of the first homology of $L(k,l)$ we have two generators $e_1\cong (1,0)$ and $e_2 \cong (0,1)$, coming from the total space $T^2  \times [0,1]$,  and the relations 
  $e_2=0$ and $ke_1-le_2=0.$ Thus, $ke_1=0$, that is, $e_1$ quotients to the generator of the cyclic group $\mathbb Z_k $, and, in particular,  $H_1(L(k,l),\mathbb Z )=\mathbb Z_k. $ However, since the transversal orbit $L$ is obtained by collapsing  $T^2\times \{0\}$ along the circle of slope $(0,1)$, it corresponds to the generator $e_1$ in the first homology of $L(k,l).$ In particular, it is non-vanishing, and so is its Poincar\'e dual class.

If $Y=S^3$, because
$H^2(S^3, \mathbb Z)=0$, we are not able to use the same argument. In this case, we use another obstruction to distinguish $\xi_1$ and $\xi_2$.

 \textbf{Step 3.} \emph{Obstruction class $ d_3(\xi_1,\xi_2)$ on $S^3.$}
 
Since $d_2(\xi_1, \xi_2) = 0$ on $S^3$, then  we employ an obstruction class
$$d_3(\xi_1, \xi_2) \in H^3(Y;\pi_3(S^2)) \cong H^3(Y;\mathbb Z)=\mathbb Z $$
to check if $\xi_1$ and $\xi_2 $ are homotopic as plane fields over all of $S^3$. We skip the formal definition (see  \cite[Section 4.2.3.]{Geiges}) and, in order to avoid specific trivialisations and induced framings, we make use of the following characterisation by Gompf.

 \begin{definition} \cite[Definition 4.2.]{Gompf} Let $ \xi$ be an oriented 2-plane field on a closed, oriented 3-manifold $Y$ (not necessarily connected) such that the first Chern class of $ \xi$ is a torsion class. Suppose that $(Y,  \xi)$ is the almost-complex boundary of a compact, almost-complex 4-manifold $(X,J)$, that is $ \partial X = Y$ (as an oriented manifold) and $ \xi$ is the field of complex lines in $TY \subset TX|_Y $. Then, define 
 $$ \theta ( \xi ) = (PD (c_1(X)))^2 -2 \chi(X)- 3 \sigma(X)  \in \mathbb Q$$

 and
 $$d_3(\xi_1,\xi_2)=\frac{1}{4}\theta(\xi_1,\xi_2),$$ 
 where $c_1(X)$ is the first Chern class of $(TX, J)$, $ \chi(X)$ is the Euler characteristic  of $X$ and $ \sigma$ is the signature of  $X$.

       \end{definition}      
       
       According to       \cite[Theorem 4.5.]{Gompf} the number $      \theta$ is an invariant of $(Y,  \xi)$ that does not depend on the choice of $(X,J).$ Moreover,  according to
         \cite[Theorem 4.16.]{Gompf} it follows that 
          the contact structures    $\xi_1$ and $   \xi_2$ are not  homotopic as 2-plane fields  if     $ \theta ( \xi_1 )\neq\theta ( \xi_2 ).$
         
                        \vskip2mm

    In order to employ $  \theta$, let us briefly explain that         $(S^3, \xi_1 )$ and $(S^3, \xi_1 )$ can be realised as  concave contact boundaries of the linear plumbings $(0,0,0,0,0,0)$ and $(0,0,0,0)$, respectively. Note first that the contact (toric) manifolds 
    $(S^3, \xi_1 )$ and $(S^3, \xi_2)$ are defined by the numbers $t_1=0, t_2= 5\pi/2$ and 
        $t_1=0, t_2= 3\pi/2,$ respectively. 
          As described in 
          Section \ref{sec:mnrstw}, we decompose the linear plumbing
          $(0,0,0,0,0,0)$          into the sequence of five  equal linear plumbings $(0,0)$ and we glue them via map $\begin{bmatrix}
 0 & -1 \\
1 &  0
\end{bmatrix}.$ Since the concave contact boundary of the linear plumbing $(0,0)$ is defined by the numbers $t_1=0$ and $t_2=\pi/2$, after gluing five copies of it, we obtain the linear plumbing $(0,0,0,0,0,0)$ whose concave contact boundary is defined by the numbers 
$t_1=0$ and $t_2=5\pi/2$. 
 Similarly,  $(S^3, \xi_2 )$     can be realised as a concave contact boundary of   the linear plumbing $(0,0,0,0).$ 
For more general description of how to associate a linear plumbing to any given  non-free contact toric 3-manifold
   see the proof of Theorem  \ref{thm:toric_is_plumbing}.

      Denote by $X_1$ and $X_2$  the symplectic manifolds $(0,0,0,0,0,0)$
      and $(0,0,0,0).$ Then, 
        $X_1$ and $X_2$ admit  almost complex structures
         $J_1$ and $J_2$ compatible with the corresponding symplectic forms $\omega_1$ and $\omega_2$. These almost complex structures can be chosen near the boundary in such a way that $J_kR_k=Y_k$ and $J_k \xi_k= \xi_k$, for $k=1,2,$ where $Y_k$ is a Liouville vector field corresponding to $\omega_k$ and $R_k$ is the Reeb vector field of the contact form $\iota_{Y_k}\omega_k.$ Therefore, the contact structures 
         $\xi_1$ and $\xi_2$ are invariant under the corresponding almost complex structures $J_1$ and $J_2$
         and we are able to make use of the invariant  $\theta ( \xi )$. Let us compute all the invariants involved in the definition of $\theta.$
         
      \begin{itemize}
            \item Signature of the plumbing $X$ is equal to
            $$      \sigma(X)= b_2^+-b_2^-, $$
            where $b_2^+$ and $b_2^-$ denote the number of positive and negative eigenvalues of the corresponding intersection form $Q$ of the given plumbing (see \cite[Section 1.2.]{GS} for the general definition of the intersection form on a 4-manifold).
            
            In our case of  linear plumbings over spheres, the intersection forms  for $X_1$ and $X_2$ are respectively given by
            $$Q_1=\left[\begin{array}{cccccc} 
0 & 1 & 0 & 0 & 0 & 0 \\
1 & 0 & 1 & 0 & 0& 0 \\
0 & 1 & 0 & 1 & 0 & 0 \\
0 & 0 & 1 & 0 &     1  & 0  \\
 0 &0    &0    &        1  & 0& 1\\
 0 &  0& 0  &        0      & 1 & 0
\end{array}  \right] , Q_2=\left[\begin{array}{cccc} 
0 & 1 & 0 & 0  \\
1 & 0 & 1 & 0  \\
0 & 1 & 0 & 1  \\
0 & 0 & 1 & 0      
\end{array}  \right] .$$
The eigenvalues are the solutions $\lambda$'s of equations $\det(Q_k-\lambda I)=0$, $k=1,2,$ that is,   $\lambda^6-5 \lambda^4+6 \lambda^2-1=0$ and $\lambda^4-3\lambda^2+1=0$. We solve these equations by plugging $\lambda^2=\Lambda>0$ and we obtain that the number of positive and negative eigenvalues is equal for both matrices. Thus, 
$$   \sigma(X_1)=    \sigma(X_2)=0.$$
            
                  \item The Euler characteristic of the 4-manifold $X$ is equal to
                  $$\chi(X) =b_0-b_1+b_2-b_3+b_4,$$
                  where $b_j$ is the $j$th Betti number, i.e. the rank of $H_j(X,\mathbb Z).$
                  
                  Since we deal only with the plumbings over spheres, these can be obtained by attaching $k$ 2-cells along one 0-cell, where $k$ denotes the number of spheres in the base. See 
                   \cite[Example 4.6.2]{GS}.                  
                  Therefore,  $b_0=1,$ $b_1=b_3=b_4=0$ and $b_2=k$ and 
                  $$\chi(X_1) =7, \chi(X_2) =5.$$                  
                                    
                  \item The Poincar\'e dual of the first Chern class $PD(c_1(X))$ will be computed using the adjunction formula (\cite[Theorem 1.3]{McD}), which says
                  that for a symplectic 4-manifold $X$ and a symplectic submanifold $C$ (one can choose a compatible almost complex structure $J$ such that $J$ preserves the tangent bundle of $C$) the following equality holds
                  $$   \langle c_1(X),  [C]  \rangle=2-2g(C)+[C]^2, $$     
                  where $g(C)$ is the genus and $[C]^2$ the self-intersection number of $C.$

       Therefore, denoting by $C_j$ the spheres in the base of the plumbing $X_i$, we obtain      
                    $$   \langle c_1(X_i),  [C_j]  \rangle=2.$$      
                    
           We further present the computation in the case of the plumbing $X_1.$ Since $PD(c_1(X_1))$ is an element in the second homology of $X_1$, whose generators are the classes represented by the spheres in the base of the plumbing, set
             $$PD(c_1(X_1))= \sum_{j=1}^6 a_j[C_j],$$
             where $[C_j]=[(0, ..., 1, ..., 0)]  \in H_2(X_1, \mathbb Z)$.
             Let us find $a_j,$ for all $j=1,    \ldots, 6$. By pairing  the classes $[C_j]$, $j=1,    \ldots, 6$ we obtain
               $$ [C_i] Q_1 [C_j]=1,$$
               if $i=2,j=1$, or $i=1,3, j=2$, or $i=2,4,j=3$, or $i=3,5, j=4,$ or $i=4,6,j=5,$ or $i=5,j=6,$
                and 
                 $$ [C_i] Q_1 [C_j]=0,$$
                 otherwise.                    
             
             Thus,
             $$2=   \langle c_1(X_1),  [C_j]  \rangle  = PD(c_1(X_1)) Q_1 [C_j]=
                \begin{cases}
                a_2, & j=1,  \\
                a_1+a_3,& j=2, \\       
                  a_2+a_4, &j=3, \\
                   a_3+a_5, & j=4,  \\        
                   a_4+a_6, & j=5, \\
                                   a_5, & j=6.
                 \end{cases}   $$
                    
               Thus, $a_1= a_2=a_5=a_6=2,$ $a_3=a_4=0$ and
                $$PD(c_1(X_1))= 2[C_1]+2[C_2]+2[C_5]+2[C_6].$$
               Therefore,
                                $$(PD(c_1(X_1)))^2=PD(c_1(X_1)) Q_1PD(c_1(X_1)) =16.$$
                           
                   Analogously,
                                $$(PD(c_1(X_2)))^2=8.$$

               \end{itemize}      
               
               Finally,
               $$   \theta(   \xi_1)=2,  \theta (   \xi_2)=-2,$$
               and we conclude that the contact structures $   \xi_1$ and $   \xi_2$ are not contactomorphic.
                     
The proof of Theorem  \ref{thm:overtwisted} is completed.

\end{proof}

%

\begin{remark}  $S^1 \times S^2$ is obtained from $T^2\times [0,1]$ by collapsing  the  tori $T^2\times \{0\}$ and $T^2\times \{1\}$ along circles of linear slopes $v_0$ and $v_1$. 
That is, $v_1=v_0$ or $v_1=-v_0.$ Following the construction by Lerman, we conclude that in the first case the moment cone is defined by $t_1=0, t_2=(2n-1)\pi$, $n\geq1,$ while in the second case the moment cone is defined by the numbers $t_1=0, t_2=2n\pi,$ $n\geq1.$ 

If $t_1=0, t_2= \pi$, then the moment cone is the half-plane and it corresponds to the unique tight contact structure given as the kernel of the contact form
$\alpha_{st}=hdt+\frac{i}{4}(zd\bar{z}-\bar zdz).$ See Section \ref{sec:tight} for the details.

If $t_1=0, t_2=2\pi,$ then the moment cone is the whole plane and it corresponds to the contact structure given as the kernel of the following contact form
$$\alpha_{ot}=-(1-3\cos 2\theta) dt-\sqrt{6} \cos\theta \sin 2\theta d\phi,$$
where $t \in S^1$ and $(\theta,\phi) \in [0,\pi] \times \mathbb R/2\pi\mathbb Z$, with a toric action
that rotates $t$ and $\phi$ coordinates. 
 The moment map is
 $$H(t, \theta,\phi)=-(1-3\cos 2\theta, \sqrt{6} \cos\theta \sin 2\theta )$$
 
 and the moment map image is a closed curve. Therefore, the moment cone is the whole space and the given contact structure has to be overtwisted.
 
 From the classification of contact toric structures it follows that the contact structure $\ker\alpha_{ot}$ is obtained by performing the half-Lutz twist to the unique tight contact structure $\ker \alpha_{st}$ on $S^1 \times S^2$.

We remark that the contact structure $\ker\alpha_{ot}$ is previously introduced by Taubes in \cite{Tau}, where he observed pseudoholomorphic curves on the symplectization of this contact structure.

\end{remark}

\section{Proof of Theorem \ref{thm:toric_is_plumbing}} \label{sec:toric_is_plumbing}
Our goal in this section is to prove Theorem~\ref{thm:toric_is_plumbing}, that every contact toric 3-manifold with non-free action can be realized as the concave boundary of a symplectic linear plumbing of disk bundles over spheres. By Theorem~\ref{theorem Lerman class}, it suffices to show that every $(Y(t_1,t_2),\xi_{t_1,t_2})$ can be realized as such a concave boundary, and by Remark~\ref{rem:lens_space}, we may assume $t_1=0$. We will prove this first for a subset of possible values for $t_2$ and gradually build up to the general case.

\begin{lemma}\label{lemma:algorithm}
A contact toric 3-manifold 
$(Y(t_1,t_2),\xi_{t_1,t_2})$, where
 $t_1=0$ and  $0<t_2\leq  \pi/2$  can be realised as a concave boundary of some linear plumbing $(s_1, \ldots, s_n)$.
\end{lemma}
\begin{proof}
If $t_2=\pi/2$ then the moment cone is spanned by the rays pointing out of the origin $(1,0)$ and $(0,1)$. By performing the transformation
$\begin{bmatrix}
 -1 & 0 \\
0 &  -1
\end{bmatrix}$ we obtain the rays given by the equation (\ref{eq:rays}) for $s_1=s_2=0.$ 
Therefore, the corresponding contact toric manifold can be realised as a concave boundary of the plumbing $(0,0).$ Note that this is precisely the standard contact toric sphere $S^3.$

Suppose $t_2< \pi/2.$ Denote the corresponding ray by $(l,k),$ where $l,k>0.$ Then the rays $(1,0)$ and $(l,k)$ bound a convex moment cone (and the corresponding manifold is diffeomorphic to a lens space $L(k,l)$).
As explained in the proof of \cite[Theorem 5.1.]{MNRSTW}, if
 \begin{equation} \label{eq:con_fraction}
  k/l=s_1-\frac{1}{s_2-\frac{1}{\cdots-\frac{1}{s_n}}},
 \end{equation} 
then the moment cone corresponding to the boundary of the plumbing $(s_1, \ldots, s_n)$ is spanned by the rays $(1,0)$ and $  \pm (l,k).$
 Moreover, according to \cite[Theorem 5.3.]{MNRSTW}), if
$s_1 \geq0$ and $s_2,\ldots, s_n \leq-2$ then the associated moment cone is convex  and thus, the contact structure on the boundary is tight. This condition is relevant since $-k/-l=k/l,$
however, the rays $(1,0)$ and $(-l,-k)$ bound a concave moment cone, and,  the rays $(1,0)$ and $(l,k)$ may also span an angle  $t_2>2 \pi $. 
Therefore, it is enough to find numbers $s_1,\ldots, s_n$ such that relation (\ref{eq:con_fraction}) holds, where $s_1 \geq0, s_2,\ldots, s_n \leq-2.$

Suppose $l=1$. Then, the contact toric structure on $L(k,1)$ is determined by the moment cone that is spanned by the rays $(1,0)$ and $(1,k)$. The corresponding contact toric manifold can be realised  as the boundary of the plumbing $(k-1,-1).$ Namely, the rays of the corresponding L-shape are given by the directions $(-1,k-1)$ and $(-1,-1),$
pointing out of the origin, and the transformation
$\begin{bmatrix}
 -1 & 0 \\
1-k &  -1
\end{bmatrix}$ maps these rays to the rays $(1,0)$ and $(1,k).$

Suppose $l>1.$ We now present an algorithm to find suitable numbers $s_1,\ldots, s_n. $ Note that we have slightly different requirements than typical continued fraction expansions because we require $s_1\geq 0$ and $s_2,\dots, s_n\leq -2$.

Define
$$s_1= \lfloor{k/l}\rfloor,$$
where $\lfloor{ \cdot}\rfloor, $ denotes the integer part of the number. Since $k,l>0$, obviously $s_1\geq0. $

Next, set $r_1=k/l-s_1.$ Then $0 < r_1<1$ and 
$$k/l=s_1- \frac{1}{-1/r_1}.$$
Define
$$s_2= -\lfloor{1/r_1}\rfloor-1.$$ 
Since $1/r_1>1$, it holds $s_2\leq-2.$ Set $r_2=-1/r_1-s_2.$ If $r_2=0$ we are done. Otherwise, $0<r_2<1$ and
$$-1/r_1=s_2-\frac{1}{-1/r_2}.$$
We continue inductively by defining
$$s_{j+1}= -\lfloor{1/r_{j}}\rfloor-1,$$
for all $j=2, \ldots, $ where  $r_{j}=-1/r_{j-1}-s_j.$ 
Analogously as above ,  $s_{j+1}\leq-2$ and $0  \leq r_j<1.$
The process  terminates if there exists $n\in \mathbb N $ such that $r_j>0,$ for all $j=1,\ldots, n-1,$ and $r_n=0,$ i.e.   $1/r_{n-1}$ is an integer number. In that case we define
$$s_n=-1/r_{n-1}.$$
 Since $r_{n-1}<1,$ then $s_n \leq-2.$
 
It is now left to prove that this process terminates.
Denote by $1/r_j=k_j/l_j.$ Obviously $k_j,l_j>0.$ Let us show $l_{j+1}<l_j,$ for all $j\geq0, $ where $l_0=l.$ First, $l_{1}$ is precisely equal to $k \mod l$. The rest is always strictly less than the denominator, therefore 
$l_{1}<l.$ Next, for $j>0$, $l_{j+1}$ is equal to $l_j- (k_j  \mod l_j).$ This is because $s_{j+1}= -\lfloor{1/r_{j}}\rfloor-1.$ Therefore, $l_{j+1}<l_j,$ for all $j\geq0.$ Since $l_j>0$ and it is always an integer number it follows that  
 there exists  $n\in \mathbb N $ such that $l_{n-1}=1.$ Then 
$1/r_{n-1}$ is an integer number and $s_n$ is also well define.

 The reader can check that the relation (\ref{eq:con_fraction}) holds.
\end{proof}

\begin{lemma}\label{lemma:sharp}
Suppose $t_1=0$ and $\pi/2<t_2<\pi$. Then, there exists $0<t_2'<\pi/2$ such that $(Y(t_1,t_2),\xi_{t_1,t_2})$ is equivariantly contactomorphic to  $(Y(t_1',t_2'),\xi_{t_1',t_2'})$, for  $t_1'=0$.
\end{lemma}
\begin{proof}
It is enough to find an $SL(2, \mathbb Z)$ transformation between the corresponding cones.
The rays of the moment cone defined by the numbers $t_1=0$ and $\pi/2 <t_2<\pi$ are $(1,0)$ and $(-m,n)$, for some $m,n>0.$ Then, the transformation
$\begin{bmatrix}
 1 & k \\
0 &  1
\end{bmatrix},$ where $k>0$ satisfies $kn-m>0$ maps the rays $(1,0)$ and $(-m,n)$ to the rays $(1,0)$ and $(-m+kn,n)$, respectively. The angle between the positive part of the $x$-axis and the later ray 
is $0<t_2'< \pi/2.$
\end{proof}

\begin{lemma}\label{lemma:two_zeros}
Suppose that $(Y(t_1,t_2),\xi_{t_1,t_2})$
is realised as a concave contact toric boundary of the linear plumbing $(s_1, \ldots, s_n)$ via~\cite[Theorem 4.1]{MNRSTW}. Then
$(Y(t_1,t_2+\pi),\xi_{t_1,t_2+\pi})$
can be realised as a concave contact boundary of the plumbing $(s_1, \ldots, s_n,0,0)$.
\end{lemma}
\begin{proof}  
Suppose the linear plumbing $(s_1,\ldots, s_n)$ is obtained via the construction described in Section~\ref{sec:mnrstw}, by gluing together $n-1$ L-shapes corresponding to pairs
$$(s_1,0),\ldots, (s_{i-1},0),(s_i,s_{i+1}),(0,s_{i+2}),\ldots, (0,s_n)$$
where $s_i\geq 0$.
We decompose the linear plumbing $(s_1, \ldots, s_n,0,0)$ by gluing together $n+1$ L-shapes 
\begin{equation}\label{sequence of L-shapes}
(s_1,0),\ldots, (s_{i-1},0),(s_i,s_{i+1}),(0,s_{i+2}),\ldots, (0,s_n), (0,0), (0,0).
\end{equation}
This decomposition satisfies the requirements to apply the construction of~\cite[Theorem 4.1]{MNRSTW}, since in every pair $(a,b)$ at least one of $a$ and $b$ is non-negative, thus we can perform the gluing to obtain the plumbing with concave boundary. We compare the moment image of the plumbing $(s_1,\dots, s_n)$ with that of $(s_1,\dots, s_n,0,0)$.

The last L-shape, $L_{n+1}$ corresponding to the pair $(0,0)$ is glued to the next L-shape $L_n$ (also corresponding to the pair $(0,0)$), via $A_{n+1}=\begin{bmatrix}
0 & -1 \\
1 & 0 
\end{bmatrix}$. $L_n\cup A_{n+1}(L_{n+1})$ has concave contact boundary spanned by the rays $(1,0)$ and $(-1,0).$ In particular, the determinant of these rays is zero and the angle between them is $\pi$. After performing all the gluings, the moment image of the plumbing $(s_1,\dots, s_n,0,0)$ will be obtained by gluing $A_2\ldots A_n(L_n\cup A_{n+1}(L_{n+1}))$ to the moment image of the plumbing $(s_1,\dots, s_n)$.

In the final moment map image, the angle between the rays that bound the piece $A_2\cdots A_n(L_n\cup A_{n+1}(L_{n+1}))$ is still $\pi,$
as linear transformations preserve determinant.

\end{proof}

\begin{proof}[Proof of Theorem \ref{thm:toric_is_plumbing}]
By Theorem~\ref{theorem Lerman class}, any contact toric 3-manifold is $(Y(t_1,t_2),\xi_{t_1,t_2})$ for some $0\leq t_1<2\pi$, $t_2>t_1$, with $\tan(t_i)$ rational when defined.
We may perform suitable $SL(2, \mathbb Z)$ transformation to get $t_1=0,$ i.e. that the first ray of the moment cone is equal to the positive part of the $x$-axis. We now divide the proof into the following cases.

\begin{itemize}

 \item If $t_2 \leq \pi/2$ then we apply Lemma \ref{lemma:algorithm}
and therefore $(Y,\xi)$ can be realised as a concave contact boundary of the linear plumbing $(s_1,\ldots, s_n).$

\item  If $\pi/2 <t_2<\pi$ then, according to Lemma \ref{lemma:sharp}, we may assume that $t_2<\pi/2$ and then apply  Lemma \ref{lemma:algorithm}.

\item If  $t_2=\pi $ then $(Y, \xi)$ can be realised as a concave boundary of the plumbing  $(0,0,0).$

\item If $t_2> \pi$ then there exists $t_2'<\pi$ such that $t_2=t_2'+k\pi,$ for some $k\geq1. $ Then, as explained above, the contact toric manifold classified by $t_1=0$ and $t_2'$ can be realised as a concave contact boundary of some linear plumbing $(s_1, \ldots, s_n).$
We now inductively apply Lemma \ref{lemma:two_zeros} and conclude that the contact toric manifold classified by $t_1=0$ and $t_2=t_2'+k \pi$ can be realised as a concave boundary of the linear plumbing
$(s_1\ldots, s_n, \underset{2k}{\underbrace{0,\ldots,0}}).$ 

\end{itemize}
\end{proof}

 \begin{remark}
 Note that the linear plumbing whose boundary is a certain contact toric manifold is not unique. For instance, a blow up of the intersection point of the two adjacent spheres with self-intersection numbers $s_i$ and $s_{i+1}$ in the base of the plumbing 
changes $(s_1,  \ldots, s_i,s_{i+1},  \ldots, s_n)$  into $(s_1,  \ldots, s_{i}-1,-1,s_{i+1}-1,  \ldots, s_n).$
The corresponding contact toric structure on the boundary remains the same.
For more details on the topology of this transformation we refer to \cite{Neumann}. For the toric description, we refer to \cite[Example 5.7]{MNRSTW}.

  \end{remark}


\begin{thebibliography}{9}

 \bibitem[AM12]{AM} 
\uppercase{Abreu, M., Macarini, L.} 
 \textit{ Contact homology of good toric contact manifolds.} Compositio Math, Vol 148,  (2012), 304-334.
 
 \bibitem[BM92]{BM92} 
\uppercase{Banyaga, A., Molino, P.}
  \textit{G\'eom\'etrie des formes de contact compl\'etement int\'egrables de type toriques.} 
 S\'eminaire Gaston Darboux de G\'eom\'etrie et Topologie Diff\'erentielle, 1991–1992,  Univ. Montpellier II, Montpellier (1993), 1–25.

 \bibitem[BM96]{BM96} 
\uppercase{Banyaga, A., Molino, P.}
 \textit{ Complete integrability in contact geometry.} Penn State preprint PM 197, (1996).


 
  \bibitem[BG00]{BG} 
\uppercase{Boyer, C. P., Galicki, K.}  
 \textit{ A note on toric contact geometry.} J. of Geom. and Phys. Vol 35,  (2000) 288–298.

     

 
 \bibitem[Del88]{Delzant} 
\uppercase{Delzant, T.}
  \textit{Hamiltoniens p\'eriodiques et image convexe de l’application moment.} 
  Bull. Soc. Math. France, Vol 116, (1988), 315–339. 
  
  
   \bibitem[Eli89]{Eliashberg} 
\uppercase{  Eliashberg, Y. }
    \textit{ Classification of overtwisted contact structures on 3-manifolds.} Invent Math, Vol  98, (1989) 623–637.   
    
  
  \bibitem[Gei09]{Geiges} 
\uppercase{Geiges, H.}
  \textit{Introduction to Contact Topology}
  Cambridge University Press (2009).
  
  
   \bibitem[Gom98]{Gompf} 
\uppercase{ Gompf, R.}
  \textit{ Handlebody construction of Stein surfaces.}, Ann. of Math. Vol 148, (1998), 619–693.  
  
  
     \bibitem[GS99]{GS} 
\uppercase{ Gompf, R., Stipsicz, A.}
  \textit{ 4-manifolds and Kirby calculus.}, Graduate Studies in Mathematics
Vol 20,
American Mathematical Society (1999).    


     
   \bibitem[H00]{Honda} 
  \uppercase{Honda, K.} 
    \textit{On the classification of tight contact structures I.} 
     Geom. Topol. Vol 4, (2000), 309-368.
          

 \bibitem[K97]{Kanda} 
  \uppercase{Kanda, Y.} 
    \textit{The classification of tight contact structures on the 3–torus.} 
     Comm. in Anal. and Geom. Vol 5, (1997), 413–438.
          



    
  \bibitem[Ler01]{Lerman1} 
\uppercase{Lerman, E.}
  \textit{Contact cuts.} 
   Israel J. Math. Vol 124, (2001),  77–92. 
 
 
     
  \bibitem[Ler03]{Lerman2} 
\uppercase{Lerman, E.}
  \textit{Contact toric manifolds.}  J. Symplectic Geom. Vol 1, (2003),  785–828. 
  
       
       \bibitem[Ma15]{M15} 
\uppercase{Marinkovi\'c, A.}
  \textit{Fillability of contact toric manifolds.} Period. Math. Hung. Vol 73, (2016), 16–26. 

  \bibitem[MNRSTW25]{MNRSTW} 
\uppercase{Marinkovi\'c, A., Nelson, J., Rechtman, A.,Starkston, L., Tanny, S., Wang, L.}:
  \textit{Properties of contact toric structures and concave boundaries of linear plumbings}, arXiv preprint arXiv:2501.08451.

       \bibitem[McD91]{McD} 
\uppercase{McDuff, D.}
  \textit{The local behaviour of holomorphic curves in almost complex 4-manifolds.} J. Differential Geometry Vol 34, (1991), 143–164. 


  
         \bibitem[Mi65]{Milnor} 
\uppercase{  Milnor, J. W.} 
 \textit{Topology from the Differentiable Viewpoint.} The University Press
of Virginia, Charlottesville (1965).       
       

  
        
 \bibitem[Neu81]{Neumann} 
 \uppercase{Neumann. W. D.}
   \textit{A Calculus for Plumbing Applied to the Topology of Complex Surface Singularities and Degenerating Complex Curves} 
Transactions of the AMS
Vol. 268,  (1981),  299-344.        
      
 \bibitem[T02]{Tau} 
 \uppercase{ Taubes, C. H. }
   \textit{A compendium of pseudoholomorphic beasts in $ \mathbb R  \times S^1 \times S^2$.} 
  Geom. Topol. Vol 6, (2002) 657-814.  
       
     



\end{thebibliography}
\end{document}